\documentclass[12pt]{article}

\usepackage{amssymb,amsmath,amsthm}
\usepackage[hidelinks]{hyperref} 

\theoremstyle{plain}
\newtheorem{theorem}{Theorem}
\newtheorem{corollary}[theorem]{Corollary}
\newtheorem{lemma}[theorem]{Lemma}
\newtheorem{proposition}[theorem]{Proposition}

\theoremstyle{definition}

\theoremstyle{remark}
\newtheorem{remark}[theorem]{Remark}

\DeclareMathOperator{\NN}{\mathbb{N}}

\begin{document}

\begin{center}
\vskip 1cm{\Large \bf Some Double Sums Involving Ratios of Binomial Coefficients Arising From Urn Models}
\vskip 1cm

David Stenlund \\
Mathematics and Statistics \\
\r{A}bo Akademi University \\
FI-20500 \r{A}bo \\
Finland \\
\href{mailto:david.stenlund@abo.fi}{\tt david.stenlund@abo.fi} \\
\ \\
James G.~Wan \\
Engineering Systems and Design \\
Singapore University of Technology and Design \\
8 Somapah Road, 487372 \\
Singapore \\
and \\
School of Mathematical and Physical Sciences \\
The University of Newcastle \\
University Drive, Callaghan NSW 2308 \\
Australia \\
\href{mailto:james_wan@sutd.edu.sg}{\tt james{\_}wan@sutd.edu.sg} \\
\end{center}

\vskip .2 in

\centerline{\bf Abstract}
In this paper we discuss a class of double sums involving ratios of binomial coefficients. The sums are of the form
\[ \sum_{j=0}^{n} \sum_{i=0}^j \frac{\binom{f_1(n)}{i}}{\binom{f_2(n)}{j}}\,c^{i-j}, \]
where $f_1, f_2$ are functions of $n$. Such sums appear in the analyses of the Mabinogion urn and the Ehrenfest urn in probability. Using hypergeometric functions, we are able to simplify these sums, and in some cases express them in terms of the harmonic numbers.

\section{Introduction}

Sum identities involving binomial coefficients frequently arise in combinatorics, number theory and probability.
There are well-known methods \cite{Gould1, Gould2, Summa, Wilf, Riordan} for evaluating single sums of a product of binomial coefficients, ranging from combinatorial interpretations \cite{sved} to generating functions \cite{gf}. Single sums over a reciprocal of a binomial coefficient have also been studied \cite{Belbachir, Borwein, Gould3, Sofo, Zhao}, where a standard technique is to express $\binom{n}{k}^{-1}$ as a Beta integral.

In this paper, we focus on \textit{finite double sums} whose summand is a \textit{ratio} of binomial coefficients. Such sums do not seem to occur extensively in the literature. Of particular interest are the two identities
\begin{align}
\sum_{j=0}^{n} \sum_{i=0}^j \frac{\binom{2n+2}{i}}{\binom{2n+1}{j}} & = (n+1) \sum_{k=0}^{n} \frac{1}{2k+1}, \label{Wansum} \\
\sum_{j=0}^{n} \sum_{i=0}^j \frac{\binom{2n+1}{i}}{\binom{2n}{j}} & = \Big(n+\frac{1}{2}\Big) \sum_{k=0}^{n} \frac{1}{2k+1} + \frac{2^{2n-1}}{\binom{2n}{n}}, \label{WansumB}
\end{align}
both of which are established in Corollary \ref{cor_main}. 

Our interest in these equations originates from studying a stochastic process known as the \textit{Mabinogion urn model} \cite[pp.~159--163]{Williams}. In this model, an urn contains some white and black balls; at each time step, a ball is drawn at random and its color noted; it is then returned to the urn while a ball of the opposite color (if there is any left) has its color switched. The first author studied the expected time to absorption of the process \cite{Stenlund}, and the treatment included solving the non-homogeneous recurrence relation
\begin{equation*}
X(k) = \frac{n-k}{n}\, X(k-1) + \frac{k}{n}\, X(k+1) + 1,
\end{equation*}
with boundary conditions $X(0)=X(n)=0$. The solution contains a double sum of binomial coefficient ratios, and in the special case when $n=2k$ it can be simplified using the result \eqref{WansumB}. The right hand side of  \eqref{WansumB}, for instance, facilitates the analysis of the asymptotic behaviour of the solution. 

Similar expressions are also found in the related \textit{Ehrenfest urn model} used in statistical mechanics \cite{Flajolet}. When starting with $2n+2$ black balls and $0$ white balls, the expected number of steps until there is the same number of balls of each color is given by \eqref{Wansum}. 

We note here that virtually all binomial identities, including the ones given in this paper, can be verified on a computer using creative telescoping, for instance with the Wilf-Zeilberger algorithm \cite[Ch.~6--7]{Wilf} and its extensions. Indeed, our first complete proof of the identity \eqref{Wansum} used the multivariate Celine's algorithm. However, we think that it is of interest to demonstrate a more self-contained and classical proof that can be followed step by step\,---\,especially since we applied the same method to discover and prove other identities as well. 

In Section \ref{sec_prelim}, we introduce the tools we use. In Section \ref{sec_hyper}, we present a hypergeometric proof of our main results, of which \eqref{Wansum} and \eqref{WansumB} are very special cases.  In Section \ref{sec_gen}, we take a closer look at some of our more interesting formulas and offer several extensions.

\section{Preliminaries}\label{sec_prelim}

The Gaussian \textit{hypergeometric function} ${}_2F_1$ \cite[Ch.~15]{Abram} \cite[Ch.~2 \& 3]{AAR} is defined as 
\begin{equation} \label{2f1def}
{}_2F_1(a,b;c;z) = \sum_{n=0}^\infty \frac{(a)_n (b)_n}{(c)_n} \frac{z^n}{n!},
\end{equation}
where $(\alpha)_n$ is the Pochhammer symbol given by $(\alpha)_n = \Gamma(\alpha+n)/\Gamma(\alpha)$, and $a,b,c$ are complex numbers. The series is convergent for all $|z|<1$, and for $|z|=1$ when $\Re(c-a-b)>0$; when $|z|>1$, the function is defined by analytic continuation. When $c \in \{0,-1,-2,\ldots\}$, the series is not defined, unless $a$ or $b \in \{0, -1, \ldots, c+1\}$. 

Similarly, the generalized hypergeometric function ${}_3F_2$ is defined as
\begin{equation} \label{3f2def}
{}_3F_2(a,b,c;d,e;z) = \sum_{n=0}^\infty \frac{(a)_n (b)_n (c)_n}{(d)_n (e)_n} \frac{z^n}{n!}.
\end{equation}
Hypergeometric functions provide a natural framework for analyzing binomial sums. Some useful results for  ${}_2F_1$'s include Pfaff's transformation,
\begin{equation} \label{euler}
_2F_1(a,b;c;z) = (1-z)^{-a}\,_2F_1\left(a,c-b;c;\tfrac{z}{z-1}\right),
\end{equation}
and Gauss' theorem,
\begin{equation} \label{gauss1}
_2F_1(a,b;c;1) = \frac{\Gamma(c)\,\Gamma(c-a-b)}{\Gamma(c-a)\,\Gamma(c-b)}.
\end{equation}
When $a,b \in \mathbb{C}\backslash\mathbb{Z}^-$, $\Re(b-a)>0$ and $m \in \mathbb{Z}$, equation \eqref{gauss1} gives
\begin{equation*}
\sum_{j=m}^\infty \frac{\Gamma(j+a)}{\Gamma(j+b+1)} = \frac{\Gamma(m+a)}{(b-a)\, \Gamma(m+b)},
\end{equation*}
from which we find, for any integer $n \ge m$, 
\begin{equation} \label{eq_ratio}
\sum_{j=m}^n \frac{\Gamma(j+a)}{\Gamma(j+b+1)} =  \frac{\Gamma(m+a)}{(b-a)\, \Gamma(m+b)} -  \frac{\Gamma(n+a+1)}{(b-a)\, \Gamma(n+b+1)}.
\end{equation}
When both sides are defined, equation \eqref{eq_ratio} also holds without the restrictions on $a$ and $b$  by analytic continuation.

Applying the definition \eqref{2f1def}, we may write
\[ \sum_{i=0}^{k} \binom{n}{i} x^i = x^k\binom{n}{k} {}_2F_1\left(1,-k;n+1-k;-\tfrac{1}{x}\right). \]
Using the transformation \eqref{euler}, the above formula becomes
\begin{equation}\label{eq_binomsum_hypergeom}
\sum_{i=0}^{k} \binom{n}{i} x^i = \frac{x^{k+1}}{x+1} \binom{n}{k} \,{}_2F_1\left(1,n+1;n+1-k;\tfrac{1}{x+1}\right).
\end{equation}
We make use of equations \eqref{eq_ratio} and \eqref{eq_binomsum_hypergeom} in the proofs below.

\section{Main results}\label{sec_hyper}

The theorem below illustrates the main techniques we use, and leads to identities such as \eqref{Wansum} and \eqref{WansumB}.
\begin{theorem}\label{thm_general}
Let $n \in \NN = \{0,1,2,3,\ldots\}$, $a, b \in \mathbb{C}$ and $b \notin \mathbb{Z}^-$. Then we have
\begin{equation} 
\frac{1}{2n+b+2}\sum_{j=0}^{n} \sum_{i=0}^j \frac{\binom{2n+a+2}{i}}{\binom{2n+b+1}{j}} = \sum_{k=0}^n \frac{1}{(k+1)\binom{2k+b+2}{k+1}}\bigg(\binom{2k+a}{k}+\frac{b \,\sum_{j=0}^{k-1} \binom{2k+a}{j}}{k+b+1} \bigg). \label{eq_general}
\end{equation}
\end{theorem}

\begin{proof}
In this proof, we assume that $a$ is not a negative integer, and $a-b$ is not a negative integer or $0$. The final result \eqref{eq_general} however holds without these restrictions by analytic continuation. When $a$ is negative, any resulting binomial coefficients of the form $\binom{-m}{i}$ (with $m, i \in \mathbb{N}$) can be computed using 
\begin{equation} \label{binom_neg}
 \binom{-m}{i} = \frac{(-m)(-m-1)\cdots (-m-i+1)}{i!} = (-1)^i \binom{m+i-1}{i}.
\end{equation}
We denote the left hand side of equation \eqref{eq_general} by $A(n)$, and the summand on the right hand side of \eqref{eq_general} by $B(k)$. Using a telescoping sum, we have
\begin{equation}\label{eq_A_telescoping}
A(n) = A(0) + \sum_{k=1}^{n} \big( A(k) - A(k-1) \big).
\end{equation}
Also, $A(0) = \frac{1}{b+2} = B(0)$. Therefore, it is sufficient to prove that $A(k) - A(k-1) = B(k)$ for $k = 1, 2, 3, \ldots$. 

We start by using equation \eqref{eq_binomsum_hypergeom} with $x=1$ to express the inner sum of $A(k)$ as a $_2F_1$:
\begin{align*}
A(k) 
&= \sum_{j=0}^{k} \frac{\binom{2k+a+2}{j}}{2(2k+b+2)\binom{2k+b+1}{j}} \, _2F_1\left(1,2k+a+3;2k+a+3-j;\tfrac{1}{2}\right) \\
&= \sum_{j=0}^{k} \frac{\Gamma(2k+a+3)\Gamma(2k+b+2-j)}{2\, \Gamma(2k+b+3)\Gamma(2k+a+3-j)} \, _2F_1\left(1,2k+a+3;2k+a+3-j;\tfrac{1}{2}\right) \\
&= \sum_{j=k+1}^{2k+1} \frac{\Gamma(2k+a+3)\Gamma(b+1+j)}{2\, \Gamma(2k+b+3)\Gamma(a+2+j)}\, _2F_1\left(1,2k+a+3;a+2+j;\tfrac{1}{2}\right).
\end{align*}
Next, we expand the hypergeometric function as an infinite series using \eqref{2f1def}, and switch the order of summation:
\begin{align} \nonumber
A(k) & = \sum_{j=k+1}^{2k+1} \sum_{\ell=0}^\infty \frac{\Gamma(2k+a+3)\Gamma(b+1+j)}{2\, \Gamma(2k+b+3)\Gamma(a+2+j)} \frac{(2k+a+3)_\ell}{(a+2+j)_\ell \,2^\ell} \\ \nonumber
& = \sum_{\ell=0}^\infty \sum_{j=k+1}^{2k+1} \frac{\Gamma(2k+a+3+\ell)\Gamma(b+1+j)}{2^{\ell+1}\, \Gamma(2k+b+3)\Gamma(a+2+j+\ell)} \\
& =  \sum_{\ell=0}^\infty \frac{1}{(a-b+\ell)2^{\ell+1}}\bigg( \frac{\Gamma(2k+a+3+\ell)\Gamma(k+b+2)}{\Gamma(2k+b+3)\Gamma(k+a+2+\ell)}-1\bigg), \label{eq_Aknew}
\end{align}
where we have used equation \eqref{eq_ratio} to evaluate the $j$-sum.

Using equation \eqref{eq_Aknew} and the functional equation $\Gamma(\alpha+1)=\alpha\,\Gamma(\alpha)$, we obtain
\begin{align} \nonumber
& \ A(k)-A(k-1)   \\ \nonumber
= & \, \sum_{\ell=0}^\infty \frac{\ell(k+b+1)+(a+1)(b+1)+k(a+b+1)}{2^{\ell+1}} \frac{\Gamma(2k+a+1+\ell)\Gamma(k+b+1)}{\Gamma(2k+b+3)\Gamma(k+a+2+\ell)} \\ \nonumber
= & \,\frac{\Gamma(2k+a+1)\Gamma(k+b+1)}{\Gamma(2k+b+3)\Gamma(k+a+2)} \bigg(\frac{(2k+a+1)(k+b+1)}{4(k+a+2)} \, _2F_1\left(2,2k+a+2;k+a+3;\tfrac12\right)   \\ 
& \qquad + \frac{(a+1)(b+1)+k(a+b+1)}{2}\, _2F_1\left(1,2k+a+1;k+a+2;\tfrac12 \right)\bigg). \label{eq_needctg}
\end{align}
Now, from one of Gauss' contiguous relations \cite[Eq.~(2.5.8)]{AAR}, and the simple fact that ${}_2F_1(0,\beta;\gamma;z)=1$, we get
\begin{equation*}
{}_2F_1\left(2,\beta+1;\gamma+1;\tfrac{1}{2}\right) = \frac{2\gamma}{\beta}\Big((\beta-2\gamma+2)\,{}_2F_1\left(1,\beta;\gamma;\tfrac{1}{2}\right) + 2\gamma-2\Big).
\end{equation*}
Applying this formula to equation \eqref{eq_needctg}, we get
\begin{multline*}
A(k)-A(k-1) =\frac{\Gamma(2k+a+1)\Gamma(k+b+1)}{\Gamma(2k+b+3)\Gamma(k+a+2)}\Big((k+a+1)(k+b+1) \\ +\frac{k\,b}{2}\, _2F_1\left(1,2k+a+1;k+a+2;\tfrac12\right)\Big). 
\end{multline*}
The $_2F_1$ term can be converted back into a binomial sum using \eqref{eq_binomsum_hypergeom}. After some algebra, we verify that indeed $A(k) - A(k-1) = B(k)$, and therefore \eqref{eq_general} is true.
\end{proof}

Our next theorem expresses the sums under investigation in terms of hypergeometric functions.
\begin{theorem}\label{thm_general2}
Let $n \in \NN$, $a, b, c \in \mathbb{C}$ and $b \notin \mathbb{Z}^-$. Then we have
\begin{equation} 
\frac{1}{n+b+1}\sum_{j=0}^{n} \sum_{i=0}^j \frac{\binom{n+a+1}{i}}{\binom{n+b}{j}}\, c^{i-j} = \frac{1}{b+1}\sum_{k=0}^n \sum_{j=0}^k \frac{\binom{k+a}{j}}{\binom{k+b+1}{k}}\, c^{j-k}. \label{eq_general2a}
\end{equation}
When $a-b$ is not a negative integer and $c \notin \{-1,0\}$,
\begin{multline} 
\frac{(a-b)(c+1)}{(n+b+1)c}\sum_{j=0}^{n} \sum_{i=0}^j \frac{\binom{n+a+1}{i}}{\binom{n+b}{j}}\,c^{i-j} = \frac{\binom{n+a+1}{a}}{\binom{n+b+1}{b}}\,_3F_2\left(1,a-b,n+a+2;a-b+1,a+1;\tfrac{1}{c+1}\right) \\ -  {}_2F_1\left(1,a-b;a-b+1;\tfrac{1}{c+1}\right). \label{eq_general2b}
\end{multline}
When $a-b$ is a negative integer or $0$, and $c \notin \{-1,0\}$,
\begin{multline} 
\frac{(c+1)^{b-a+1}}{(n+b+1)c}\sum_{j=0}^{n} \sum_{i=0}^j \frac{\binom{n+a+1}{i}}{\binom{n+b}{j}} \,c^{i-j} = \frac{n+b+2}{(b+1)(c+1)}\,_3F_2\left(1,1,n+b+3;2,b+2;\tfrac{1}{c+1}\right) \\ + \psi(n+b+2)-\psi(b+1)+\log\frac{c}{c+1} - \sum_{\ell=1}^{b-a} \frac{(c+1)^\ell}{\ell}\bigg( \frac{\Gamma(n+b+2-\ell)\Gamma(b+1)}{\Gamma(n+b+2)\Gamma(b+1-\ell)}-1\bigg), \label{eq_general2c}
\end{multline}
where $\psi(\alpha) = \Gamma'(\alpha)/\Gamma(\alpha)$ is the \emph{digamma function}.
\end{theorem}

\begin{proof}
Equation \eqref{eq_general2a} is proven in the same way as \eqref{eq_general}, so we only sketch the proof. Denoting the left hand side of \eqref{eq_general2a} by $\widetilde{A}(n)$ and following the same procedure, we obtain
\begin{equation} 
\widetilde{A}(k) =  \sum_{\ell=0}^\infty \frac{c}{(a-b+\ell)(c+1)^{\ell+1}}\bigg( \frac{\Gamma(k+a+2+\ell)\Gamma(b+1)}{\Gamma(k+b+2)\Gamma(a+1+\ell)}-1\bigg). \label{eq_Ak2new}
\end{equation}
Using this, we can readily compute $\widetilde{A}(k) -\widetilde{A}(k-1)$, which, upon simplification, leads to \eqref{eq_general2a}. We skip the details as they are very similar to the proof of Theorem \ref{thm_general}.

When $a-b$ is not a negative integer or $0$, equation \eqref{eq_Ak2new} can be split up into two sums, each expressible as a hypergeometric function due to \eqref{2f1def} and \eqref{3f2def}; the result is \eqref{eq_general2b}.

When $a-b$ is a negative integer or $0$, then the sum in \eqref{eq_Ak2new} is understood in the sense of
\[ \sum_{\ell = 0}^{b-a-1} \quad + \quad \sum_{\ell = b-a+1}^{\infty}, \]
plus the missing $\ell = b-a$ term, which is computed using
\[ \lim_{\ell \to b-a} \frac{c}{(a-b+\ell)(c+1)^{\ell+1}}\bigg( \frac{\Gamma(k+a+2+\ell)\Gamma(b+1)}{\Gamma(k+b+2)\Gamma(a+1+\ell)}-1\bigg). \]
We can evaluate the limit using L'H\^opital's rule; the digamma terms come from differentiation. By combining all the pieces, \eqref{eq_general2c} follows.
\end{proof}

\begin{remark} The sum on the right hand side of equation \eqref{eq_general2c} can be written as a combination of hypergeometric functions (cf.~Remark \ref{rem_fast}). Hence, by specializing $a$ and $b$ as functions of $n$, Theorem \ref{thm_general2} gives an essentially hypergeometric evaluation of 
\[ \sum_{j=0}^{n} \sum_{i=0}^j \frac{\binom{f_1(n)}{i}}{\binom{f_2(n)}{j}}\,c^{i-j}, \] 
where $f_1$ and $f_2$ can be any functions of $n$ provided that $f_2(n)\ge n$ for the range of $n$ concerned. 
\end{remark}

\begin{remark}
The outer upper limit of the double sum does not need to be $n$. For instance, take $a=b = n/2$ and $c=1$ in equation \eqref{eq_general2c}, then let $n \mapsto 2n$. The result is
\[ \frac{2}{3n+1}\sum_{j=0}^{2n}\sum_{i=0}^j \frac{\binom{3n+1}{i}}{\binom{3n}{j}} = \psi(3n+2)-\psi(n+1)-\log(2)+\frac{3n+2}{2n+2}\,_3F_2\big(1,1,3n+3;2,n+2;\tfrac12\big).\]
The expression on the left hand side is found in the expected time to absorption of the Mabinogion urn process when a certain control strategy is adopted \cite[Section~3]{Stenlund}, namely that the number of white balls in the urn is always kept below twice the number of black balls. Other control strategies give rise to similar expressions. 
\end{remark}

For certain combinations of $a$ and $b$, Theorems  \ref{thm_general} and \ref{thm_general2} can be  simplified. 

\begin{corollary}\label{thm_general_odd}
For any $n \in \mathbb{N}$, $a, b, c\in \mathbb{C}$ and $b \notin \mathbb{Z}^-$,
\begin{align}\label{eq_general_odd}
\sum_{j=0}^{n} \sum_{i=0}^j \frac{\binom{2n+a+2}{i}}{\binom{2n+1}{j}} & = \sum_{k=0}^n \frac{n+1}{2k+1} \frac{\binom{2k+a}{k}}{\binom{2k}{k}}, \\
\sum_{j=0}^{n} \sum_{i=0}^j \frac{\binom{2n+3}{i}}{\binom{2n+b+1}{j}} & =  \sum_{k=0}^n \frac{2n+b+2}{k+b+1} \frac{\binom{2k+1}{k}+\frac{4^k\,b}{k+1}}{\binom{2k+b+2}{k+1}},  \label{eq_general_b1} \\
\sum_{j=0}^{n} \sum_{i=0}^j \frac{\binom{n+1}{i}}{\binom{n+b}{j}}\, c^{j-i}  & =  \sum_{k=0}^n \frac{n+b+1}{b+1}\frac{(c+1)^k}{\binom{k+b+1}{k}},  \label{eq_aeq0} \\
\sum_{j=0}^{n} \sum_{i=0}^j \frac{\binom{n+2}{i}}{\binom{n}{j}}\, c^{j-i+1}  & =  \sum_{k=0}^n \frac{n+1}{k+1}\left((c+1)^{k+1}-1\right).  \label{eq_aeq1}
\end{align}
\end{corollary}

\begin{proof} Setting $b=0$ in equation \eqref{eq_general}, the second term in the square brackets on the right hand side vanishes, and after some simple algebra we obtain \eqref{eq_general_odd}.

If we let $a=1$ in \eqref{eq_general}, then the $j$-sum simplifies due to the binomial theorem, and we obtain \eqref{eq_general_b1}. Note that by choosing different values of $a$ such as $0$ or $-1$, formulas similar to \eqref{eq_general_b1} can be obtained, but we omit the results.

To prove equation \eqref{eq_aeq0}, take $a=0$ in \eqref{eq_general2a} and apply the binomial theorem. 

To show that \eqref{eq_aeq1} is true, take $a=1$ and $b=0$ in \eqref{eq_general2b}, then manipulate the resulting right hand side into a telescoping sum, using the contiguous relation
\[ (n+2)\,_3F_2(1,1,n+3;2,2;x)-(n+1)\,_3F_2(1,1,n+2;2,2;x) = {}_2F_1(1,n+2;2;x), \]
With $x=\tfrac{c}{c+1}$, the $_2F_1$ essentially simplifies to the right hand side summand of \eqref{eq_aeq1}.
\end{proof}

In some cases, the single sums from Corollary \ref{thm_general_odd} can be written in terms of the \textit{harmonic numbers}, defined as
\begin{equation*}
H_n:=\sum_{k=1}^n \frac{1}{k}.
\end{equation*}
Recall a connection between the harmonic numbers and the digamma function appearing in \eqref{eq_general2c}: when $\alpha \in \mathbb{N}$, $\psi(\alpha+1) = H_\alpha - \gamma$, where $\gamma$ is the Euler-Mascheroni constant.

\begin{corollary}\label{cor_main}
For any $n\in\NN$,
\begin{align}\label{eq_doublesumA}
\sum_{j=0}^{n} \sum_{i=0}^j \frac{\binom{2n+2}{i}}{\binom{2n+1}{j}} & = (n+1) \sum_{k=0}^{n} \frac{1}{2k+1} = (n+1)\Big( H_{2n+1}-\frac{H_{n}}{2} \Big), \\
\label{odd_two}
\sum_{j=0}^{n} \sum_{i=0}^j \frac{\binom{2n+3}{i}}{\binom{2n+1}{j}} & = (n+1) \sum_{k=1}^{n+1} \frac{1}{k} = (n+1) H_{n+1}, \\
\label{eq_doublesumB}
\sum_{j=0}^{n} \sum_{i=0}^j \frac{\binom{2n+1}{i}}{\binom{2n}{j}} & = \frac{2^{2n-1}}{\binom{2n}{n}} + \Big(n+\frac{1}{2}\Big) \Big( H_{2n+1}-\frac{H_{n}}{2} \Big) , \\
\label{even_two}
\sum_{j=0}^{n} \sum_{i=0}^j \frac{\binom{2n+2}{i}}{\binom{2n}{j}} & = \frac{2^{2n}}{\binom{2n}{n}} + \Big(n+\frac{1}{2}\Big)H_n.
\end{align}
\end{corollary}

\begin{proof}
Equation \eqref{eq_doublesumA} follows from \eqref{eq_general_odd} with $a=0$.  Equation \eqref{odd_two} follow from \eqref{eq_general_odd} with $a=1$, or from \eqref{eq_general_b1} with $b=0$. 

For equation \eqref{even_two}, we need to choose $b=-1 \in \mathbb{Z}^-$, so \eqref{eq_general} does not immediately apply. However, we can refer to equation \eqref{eq_A_telescoping}: 
\[ A(n)  = A(0) + \sum_{k=1}^n \big( A(k)-A(k-1) \big) = \frac{1}{b+2} + \sum_{k=1}^n B(k), \]
where we have used the notation and the conclusion in the proof of  Theorem \ref{thm_general}. Then, by letting $b=-1$ and $a=0$, applying the duplication formula for the Gamma function \cite[Eq.~(1.5.1)]{AAR}, and simplifying, we get
\[ \sum_{j=0}^{n} \sum_{i=0}^j \frac{\binom{2n+2}{i}}{\binom{2n}{j}} = (2n+1) \bigg(1+\frac{H_n}{2} -\frac{\sqrt{\pi}}{4} \sum_{k=0}^n \frac{\Gamma(k)}{\Gamma(k+3/2)} \bigg). \]
The $k$-sum can be evaluated using equation \eqref{eq_ratio}, and \eqref{even_two} follows. Equation \eqref{eq_doublesumB} is proved in a very similar manner, using $b=a=-1$.
\end{proof}

\begin{remark}
By comparing equation \eqref{odd_two} against \eqref{eq_general2b} (with $a=n+2, \ b=n+1$ and $c=1$), we obtain the hypergeometric evaluation
\[ \frac{2n+1}{n+1}\,_3F_2\left(1,1,2n+2;2,n+2;\tfrac12\right) = H_n + 2\log(2). \]
\end{remark}

\begin{remark}
The double sum
\[ \sum_{j=0}^n \sum_{i=0}^j \frac{\binom{M+1}{i}}{\binom{M}{j}}\]
gives the expected transition time from $0$ white balls to $n+1$ white balls in the Ehrenfest urn model, starting with $M+1$ black balls \cite[Eq.~(8)]{Palacios}. The cases $M=2n+1$ and $M=n+1$ have been studied and simplified using integrals \cite{Blom, Lathrop}, and correspond to our formulas \eqref{eq_doublesumA} and \eqref{eq_doublesum_whole} respectively. Hypergeometric evaluations for other values of $M$ can be readily obtained by choosing the appropriate value of $a=b$ in equation \eqref{eq_general2c}.
\end{remark}

\section{Further formulas}\label{sec_gen}

\subsection{Analysis of Corollary \ref{cor_main}}

If we fix $a \in \mathbb{Z}$  in equation \eqref{eq_general_odd}, then the right hand side simplifies as a sum of a rational function in $k$. By decomposing this rational function into partial fractions, we can express the right hand side as a combination of harmonic numbers, generalizing Corollary \ref{cor_main}.

Alternatively, we can apply the procedure below to unravel a structure behind Corollary \ref{cor_main} more clearly. The elementary identity
\[ \binom{n}{k} = \binom{n-1}{k-1}+\binom{n-1}{k} \]
gives the recursion
\begin{equation*}
\sum_{i=0}^j \binom{n}{i} = \bigg(2 \sum_{i=0}^j \binom{n-1}{i}\bigg) - \binom{n-1}{j}.
\end{equation*}
When applied iteratively, it leads to the formula (for $m\in\mathbb{N}$)
\begin{equation} \label{eq_iterative}
\sum_{i=0}^j \binom{n}{i} = 2^{m} \sum_{i=0}^j \binom{n-m}{i} - \sum_{k=1}^{m} 2^{k-1} \binom{n-k}{j}.
\end{equation}
Similarly, the reciprocals of the binomial coefficients satisfy the recursion
\begin{equation*}
\frac{1}{\binom{n}{k}} = \frac{n+1}{n+2}\Bigg( \frac{1}{\binom{n+1}{k}} + \frac{1}{\binom{n+1}{k+1}} \Bigg),
\end{equation*}
which gives
\begin{equation*}
\sum_{i=0}^k \frac{1}{\binom{n}{i}} = \frac{n+1}{n+2}\left( \bigg(2 \sum_{i=0}^k \frac{1}{\binom{n+1}{i}}\bigg) + \frac{1}{\binom{n+1}{k+1}}-1 \right).
\end{equation*}
This in turns leads to the formula (for $m\in\mathbb{N}$)
\begin{equation} \label{eq_iterative_rec}
\sum_{i=0}^k \frac{1}{\binom{n}{i}} = \frac{2^{m}(n+1)}{n+m+1} \sum_{i=0}^k \frac{1}{\binom{n+m}{i}} + \sum_{j=1}^{m} \frac{2^{j-1}(n+1)}{n+j+1}\bigg( \frac{1}{\binom{n+j}{k+1}}-1\bigg).
\end{equation}
Equations \eqref{eq_iterative} and \eqref{eq_iterative_rec} allow us to shift the upper indices of the binomial coefficients in any of our results.  For instance, combining \eqref{odd_two} with \eqref{eq_iterative}, we get 
\begin{equation*}
\sum_{j=0}^{n} \sum_{i=0}^j \frac{\binom{2n+3+m}{i}}{\binom{2n+1}{j}} = 2^{m} \sum_{j=0}^{n} \sum_{i=0}^j \frac{\binom{2n+3}{i}}{\binom{2n+1}{j}} - \sum_{k=1}^{m} 2^{m-k} \sum_{j=0}^n \frac{\binom{2n+2+k}{j}}{\binom{2n+1}{j}}.
\end{equation*}
The first sum on the right simplifies due to \eqref{odd_two}, while the rightmost inner sum is evaluable using \eqref{eq_ratio}. Hence, we obtain the following equation valid for any $n, m\in\mathbb{N}$:
\begin{equation} \label{half_odd_plus}
\sum_{j=0}^{n} \sum_{i=0}^j \frac{\binom{2n+3+m}{i}}{\binom{2n+1}{j}} = 2^{m}(n+1) \Bigg( H_{n+1} - \sum_{k=1}^{m} \frac{1}{2^{k-1}k} \bigg( \frac{\binom{2n+2+k}{n+1}}{\binom{2n+2}{n+1}} - 1\bigg) \Bigg). 
\end{equation}
Likewise, using \eqref{eq_doublesumA} and \eqref{eq_iterative}, we obtain another equation valid for any $n, m\in\mathbb{N}$:
\begin{equation} \label{half_odd_minus}
\sum_{j=0}^{n} \sum_{i=0}^j \frac{\binom{2n+2-m}{i}}{\binom{2n+1}{j}} = \frac{n+1}{2^{m+1}} \Bigg( 2H_{2n+1}-H_{n} - \sum_{k=1}^{m} \frac{2^{k+1}}{k} \bigg( \frac{\binom{2n+2-k}{n+1}}{\binom{2n+2}{n+1}} -1 \bigg) \Bigg).
\end{equation}
We may produce similar identities ad nauseam. As one more example, combining \eqref{even_two} with \eqref{eq_iterative_rec} gives, for any $m\in\{0,1,\ldots,n\}$, 
\begin{multline*} 
\sum_{j=0}^{n} \sum_{i=0}^j \frac{\binom{2n+2}{i}}{\binom{2n-m}{j}} = 2^{m-1}(2n+1-m) \Bigg(H_n + \frac{2^{2n+1}}{(n+1)\binom{2n+1}{n+1}} \\
+ \sum_{k=1}^{m} \frac{1}{2^{k-1}k} \bigg(\frac{\frac{2^{2n+2}k}{2n+2-k}-\binom{2n+2}{n+1}}{2\binom{2n+1-k}{n+1}} +1 \bigg) \Bigg).
\end{multline*}

\begin{remark}
For each fixed $m \in \mathbb{Z}$, equations \eqref{half_odd_plus} and \eqref{half_odd_minus} give the following information about the sum $\displaystyle\sum_{j=0}^{n} \sum_{i=0}^j \frac{\binom{2n+m+2}{i}}{\binom{2n+1}{j}}$:

\begin{itemize}
\item The double sum simplifies as $2^{m-1}(n+1)$ times the sum of a harmonic term and a rational function in $n$.

\item The harmonic term is $H_{n+1}$ if $m$ is a positive integer, and is $2H_{2n+1}-H_{n}$ if $m$ is a non-positive integer.

\item Analysis of the rational function furnishes asymptotics for the double sum as $n \to \infty$.
\end{itemize}
\end{remark}

\subsection{Sums with an elementary approach}

It is well-known \cite[Eq.~(2.25)]{Gould1} that
\begin{equation} \label{eq_reciprocal}
\sum_{j=0}^n \frac{1}{\binom{n}{j}} = \frac{n+1}{2^{n+1}} \sum_{k=1}^{n+1} \frac{2^k}{k}.
\end{equation}
This can be converted into a double sum using the next result.

\begin{lemma} \label{thm_symmetry}
Let $n\in\NN$ and let $f$ be a function such that $f(j)=f(n-j)$ for all $j\in\{0,1,\ldots ,n\}$. Then
\begin{equation} \label{eq_doublesum_symmetry}
\sum_{j=0}^{n} \sum_{i=0}^j \binom{n+1}{i} f(j) = 2^{n} \sum_{j=0}^{n} f(j).
\end{equation}
\end{lemma}

\begin{proof}
Splitting up the left hand side sum into two equal parts, and changing the order of indices in the second one, gives 
\begin{align*}
\sum_{j=0}^{n} \sum_{i=0}^j \binom{n+1}{i} f(j) &= \frac{1}{2}\Bigg( \sum_{j=0}^{n} \sum_{i=0}^j \binom{n+1}{i} f(j) + \sum_{j=0}^{n} \sum_{i=0}^j \binom{n+1}{n+1-i} f(n-j) \Bigg) \\
&=  \frac{1}{2}\Bigg( \sum_{j=0}^{n} \sum_{i=0}^j \binom{n+1}{i} f(j) + \sum_{j=0}^{n} \sum_{i=j+1}^{n+1} \binom{n+1}{i} f(j) \Bigg) \\
&= \frac{1}{2} \sum_{j=0}^{n} f(j) \sum_{i=0}^{n+1} \binom{n+1}{i} \; = \; 2^{n} \sum_{j=0}^{n} f(j),
\end{align*}
where we have used the binomial theorem for the last step.
\end{proof}

\begin{corollary}
For any $n \in \mathbb{N}$ and any $m \in \{0,1,\ldots, n+1\}$,
\begin{equation} \label{eq_doublesum_shift}
  \sum_{j=0}^{n} \sum_{i=0}^j \frac{\binom{n+1-m}{i}}{\binom{n}{j}}  = \frac{n+1}{2^{m+1}}\bigg( \sum_{k=1}^{n+1} \frac{2^k}{k} + \sum_{\ell=1}^m \frac{2^\ell}{\ell}\bigg). 
\end{equation} 
In particular,
\begin{equation} \label{eq_doublesum_whole}
\sum_{j=0}^{n} \sum_{i=0}^j \frac{\binom{n+1}{i}}{\binom{n}{j}} = \frac{n+1}{2} \sum_{k=1}^{n+1} \frac{2^k}{k}.
\end{equation}
\end{corollary}
\begin{proof}
Equation \eqref{eq_doublesum_whole} follows immediately from \eqref{eq_reciprocal} and \eqref{eq_doublesum_symmetry}.  Equation \eqref{eq_doublesum_shift} then follows from \eqref{eq_iterative}, \eqref{eq_doublesum_whole} and \eqref{eq_ratio} after some algebra.
\end{proof}

\begin{remark} \label{rem_fast}
Equation \eqref{eq_doublesum_whole} can also be obtained from \eqref{eq_aeq0} with $b=0$ and $c=1$, though the approach above is somewhat more elementary. Moreover, compared with \eqref{eq_general2c}, we get
\[ \sum_{k=1}^{n+1} \frac{2^k}{k} = \frac{n+2}{2}\,_3F_2\left(1,1,n+3;2,2;\tfrac12\right)+H_{n+1}-\log(2), \]
so the right hand side of \eqref{eq_doublesum_shift} can be written in terms of (fast converging) hypergeometric functions. 
\end{remark}

\subsection{A sum with many equivalences}

We find the sum below noteworthy as it admits many equivalent expressions.

\begin{proposition}
For any $n \in \mathbb{N}$, we have
\begin{align} \nonumber 
\sum_{j=0}^{n} \sum_{i=0}^j \frac{\binom{n+1}{i}}{\binom{2n+1}{j}} & = \sum_{k=0}^{\lfloor n/2 \rfloor} (-1)^k \,\frac{\binom{n+1}{2k+1}}{\binom{n}{k}} \\
& = (n+1)\sum_{k=1}^{n+1}\frac{2^k}{k \,\binom{n+1+k}{k}} \nonumber \\
& = \frac{n+1}{2^{n+1}} \sum_{k=1}^{n+1} \frac{2^k}{k} \nonumber \\
& = \frac{1}{2^n} \sum_{j=0}^{n} \sum_{i=0}^j \frac{\binom{n+1}{i}}{\binom{n}{j}}. \label{eq_1n}
\end{align}
\end{proposition}

\begin{proof}
To show the first equality in equation \eqref{eq_1n}, we set $a = -n-1$ in \eqref{eq_general_odd}, then simplify the resulting binomial coefficients with help from \eqref{binom_neg}. To establish the second equality, pick $b=n+1$ and $c=1$ in \eqref{eq_aeq0}.

For the third equality in \eqref{eq_1n}, define 
\[ C(n) : =  \frac{1}{n+1}\sum_{j=0}^{n} \sum_{i=0}^j \frac{\binom{n+1}{i}}{\binom{2n+1}{j}}. \] 
Then, using the same technique as the proof of Theorem \ref{thm_general}, we find that
\begin{equation} \label{ck_rec}
 2C(k)-C(k-1) = \frac{2}{k+1}. 
\end{equation}
Multiplying both sides of \eqref{ck_rec} by $2^{k-1}$ and summing from $k=1$ to $n$, the sum telescopes and gives
\[ 2^n C(n) - 1 = \sum_{k=1}^n \frac{2^k}{k+1}. \]
Rearranging gives the desired equality.  The final equality follows by comparison with \eqref{eq_doublesum_whole}.
\end{proof}

\begin{remark}
Many other similar sums have alternative expressions. For instance, we have
\begin{equation} \frac{1}{n+1}\sum_{j=0}^{n} \sum_{i=0}^j \frac{\binom{3n+2}{i}}{\binom{2n+1}{j}} = 2^n \Bigg(\frac34+\sum_{k=1}^{n-1}\frac{1}{2^{k}\,k}\bigg(1-\frac{5k+12}{8k+12}\frac{\binom{3k+2}{k}}{\binom{2k+1}{k}}\bigg)\Bigg). \label{eq_3n}
\end{equation}
This can be proved by denoting the left hand side as $D(n)$, showing that
\[ D(k+1)-2D(k) = \frac{2}{k}-\frac{(5k+12)}{2k(2k+3)}\frac{\binom{3k+2}{k}}{\binom{2k+1}{k}} \]
using the same  technique as the proof of Theorem \ref{thm_general}, then solving the recurrence. The right hand side of \eqref{eq_3n} is quite different from that obtained by setting $a=n$ in \eqref{eq_general_odd}.
\end{remark}

\subsection{Analogous identities with alternating sums}

We also consider some alternating versions of our double sums, which are in fact much easier. It is routine to prove by induction that
\begin{align} \label{eq_alternate1}
\sum_{i=0}^j  \binom{n}{i} (-1)^i & = (-1)^j \binom{n-1}{j}, \\  
\sum_{j=i}^m \frac{(-1)^j}{\binom{n}{j}} & = \frac{n+1}{n+2}\bigg(\frac{(-1)^m}{\binom{n+1}{m+1}}+\frac{(-1)^i}{\binom{n+1}{i}}\bigg).  \label{eq_alternate2}
\end{align}
It  immediately follows from \eqref{eq_alternate1}, for instance, that
\[ \sum_{j=0}^n\sum_{i=0}^j  \frac{\binom{n+a+1}{i}}{\binom{n+a}{j}} (-1)^i = \frac{(-1)^n+1}{2},\]
which is valid for any $a \in \mathbb{C}\backslash\mathbb{Z}^-$; compare this with equations \eqref{eq_doublesumA} and \eqref{eq_doublesum_whole}. An example of a shifted sum obtained the same way is
\[ \sum_{j=0}^n\sum_{i=0}^j  \frac{\binom{n+2}{i}}{\binom{n}{j}} (-1)^{n-i} =(n+1)\Big(H_{n+1}-H_{\lfloor\frac{n+1}{2}\rfloor}\Big),\]
compare with \eqref{eq_aeq1} with $c=-2$:
\[ \sum_{j=0}^n\sum_{i=0}^j  \frac{\binom{n+2}{i}}{\binom{n}{j}} (-2)^{j-i} =(n+1)\Big(H_{n+1}-\frac12 H_{\lfloor\frac{n+1}{2}\rfloor}\Big). \]
If we combine equations \eqref{eq_alternate1} and \eqref{eq_ratio}, then we deduce that
\begin{equation} \label{gen_alternate}
 \sum_{j=0}^n\sum_{i=0}^j \frac{\binom{n+a+1}{i}}{\binom{n+b}{j}} (-1)^{i-j} = \frac{1}{a-b-1}\bigg(\frac{\Gamma(n+a+1)\Gamma(b+1)}{\Gamma(n+b+1)\Gamma(a)}-n-b-1\bigg),
\end{equation}
which complements equations \eqref{eq_general2b} and \eqref{eq_general2c} as $c \ne -1$ there. Equation \eqref{gen_alternate} simplifies to $n+1$ when $a=b$. When $a=b+1$, the equality in \eqref{gen_alternate} is understood as a limit, so by L'H\^opital's rule,
\[  \sum_{j=0}^n\sum_{i=0}^j \frac{\binom{n+b+2}{i}}{\binom{n+b}{j}} (-1)^{i-j} = (n+b+1)\big(\psi(n+b+2)-\psi(b+1)\big). \]
This can be compared with equations \eqref{eq_aeq1} and \eqref{odd_two}.

Finally, by changing the order of summation and applying \eqref{eq_alternate2}, we produce
\begin{align*}
\sum_{j=0}^n\sum_{i=0}^j  \frac{\binom{n+1}{i}}{\binom{n}{j}} (-1)^j & = \frac{n+1}{2n+4}\big(1+(-1)^n(2^{n+2}-1)\big), \\
\sum_{j=0}^n\sum_{i=0}^j  \frac{\binom{2n+1}{i}}{\binom{2n}{j}} (-1)^j & = \frac{(-1)^n\,2^{2n-1}}{\binom{2n}{n}}+\frac{2n+1}{n+1}\frac{(-1)^n+1}{4}.
\end{align*}
Compare these results with equations \eqref{eq_doublesum_whole} and \eqref{eq_doublesumB}, respectively.

\section{Acknowledgments}

The authors are grateful to Professor Christophe Vignat for his help and for putting us in contact with each other. We are also grateful to Professor Paavo Salminen for his comments and suggestions for improvement.

\hrule \bigskip

\noindent 2010 \textit{Mathematics Subject Classification}:~Primary 05A10, Secondary 11B65, 33B15, 33C05.  \bigskip

\noindent \textit{Keywords}: binomial sum, hypergeometric function, harmonic number, Mabinogion urn, Ehrenfest urn 

\end{document}